\documentclass[a4paper,12pt]{amsart}
\newfont{\cyr}{wncyr10 scaled 1100}

\usepackage[left=2.7cm,right=2.7cm,top=3.5cm,bottom=3cm]{geometry}

\usepackage{amssymb,amsmath,latexsym,braket,amsthm, multirow,xcolor,mathrsfs}
\usepackage{amssymb,amsmath,xcolor,textcomp}
\usepackage{amsthm}
\usepackage{tikz-cd}

\usepackage{url,bbm,extarrows,setspace}
\usepackage{csquotes}
\usepackage{hyperref}
\usepackage{ amssymb }
\usepackage{tikz,subcaption}
\usepackage[enableskew]{youngtab}
\usepackage{ytableau,varwidth}
\usepackage{pst-node}
\usepackage{rotating}
\usepackage{amsfonts,amsmath,amssymb,amsthm,mathtools}
\usepackage{accents,dsfont,enumitem,graphicx}
\usepackage{array,hyperref,mathrsfs,url}
\usepackage{caption,subcaption}

\usepackage{tikz-cd} 
\usetikzlibrary{arrows}
\definecolor{purple(x11)}{rgb}{0.63,0.36,0.94}

\newtheorem{theorem}[table]{Theorem}

\newtheorem{proposition}{Proposition}[section]
\newtheorem{corollary}{Corollary}[section]
\newtheorem{example}{Example}[section]

\numberwithin{table}{section}

\newcommand{\Hp}{\mathcal{H}_p}
\newcommand{\Z}{\mathbb{Z}}

\newcommand{\Q}{\mathbb{Q}}
\newcommand{\RAC}{\MS^{\Gamma}(\mathcal{A}_{2k})}
\newcommand{\RACQ}{\MS^{\Gamma}(\mathcal{A}_{2k})^{(\bar{\Q})}}
\newcommand{\MSPoly}{\MS^{\Gamma_0(p)}(\mathcal{P}_{2k-2})}
\newcommand{\MSPolyC}{\MS^{\Gamma_0(p)}(\mathcal{P}_{2k-2}^{(\C)})}
\newcommand{\C}{\mathbb{C}}
\newcommand{\R}{\mathbb{R}}
\newcommand{\wpr}{\omega^{\prime}}
\newcommand{\Polys}{\mathcal{P}_{2k-2}}
\newcommand{\PQp}{\mathbb{P}^1(\mathbb{Q}_p)}
\newcommand{\PQ}{\mathbb{P}^1(\mathbb{Q})}

\newcommand{\T}{\mathcal{T}}

\DeclareMathOperator{\MS}{MS}
\DeclareMathOperator{\SL}{SL}
\DeclareMathOperator{\ST}{ST}
\DeclareMathOperator{\Res}{Res}

\newcommand{\Fp}{\mathcal{F}_p}
\newcommand{\tI}{w_{\infty}}

\DeclareMathOperator{\GL}{GL}
\DeclareMathOperator{\PGL}{PGL}

\newcommand{\incom}[1]{{\color{blue} { #1}}}

\newtheorem{definition}{Definition}[section]
\newtheorem{lemma}{Lemma}[section]
\theoremstyle{definition}
\newtheorem{remark}{Remark}[section]

\include{thebibliography}
\begin{document}

\title{A Shintani lift for rigid cocycles}

\author{Isabella Negrini}

\begin{abstract}
We construct a Shintani lift for rigid analytic cocycles of higher weight, attaching modular forms of half-integral weight to such cocycles. 
The expression for the Fourier coefficients of the modular form $\mathcal{RS}(J)$ attached to a cocycle $J$ is given in terms of the residues of $J$, and shares a striking similarity with the expression for the coefficients of the classical Shintani lift $\mathcal{S}(f)$ of an integral weight modular form $f$. This work aligns with the ideas of the nascent $p$-\emph{adic Kudla program} and strengthens the analogy between rigid cocycles and modular forms.
 \end{abstract}

\address{I.N., University of Toronto, 40 St. George Street, Room 6290
Toronto, ON M5S 2E4,
Canada}
\email{isabella.negrini@utoronto.ca}

\maketitle

\tableofcontents

\section*{Introduction}

\noindent

This work mirrors the classical theory of theta correspondences between automorphic forms in the new setting of rigid cocycles. As such, it fits into the emerging $p$-\emph{adic Kudla program} initiated in \cite{DGL}, one of whose goals is to build modular generating series from rigid cocycles.

Rigid cocycles were introduced in \cite{DV1} with the idea to extend the theory of complex multiplication to real quadratic fields. They were initially defined as elements of the parabolic cohomology $H^1_{par}(\Gamma, M)$, where $\Gamma:=\SL_2(\Z[1/p])$ and $M$ is a $\Gamma$-module of certain functions on Drinfeld's $p$-adic upper half-plane $\Hp$ (see Section \ref{Notation and preliminary definitions} for a detailed introduction on rigid cocycles). The values of rigid \emph{meromorphic} cocyles were conjectured in \cite{DV1} to be analogues of singular moduli for real quadratic fields. In \emph{loc. cit.} and other works, rigid cocycles behave like modular forms or at least modular functions. As an example, in \cite{DV2} certain rigid cocycles are seen as analogues of some modular functions with CM divisor by means of a rigid analytic Borcherds lift. In \cite{PhD}, some rigid \emph{analytic} cocycles play the role of integral weight modular forms in the construction of a rigid analytic Shimura lift. The present work strengthens this analogy between rigid cocycles and modular forms.

Our main result is as follows. Let $k\geq 0$ be an integer. Given a suitable weight $2k$ rigid cocycle $J$, we construct a modular form $\mathcal{RS}(J)$ of weight $k+1/2$ attached to $J$. In order to give a rough statement, consider the following set of binary quadratic forms 
$$
\Fp:=\{ax^2+bxy+cy^2|\; a,b,c \in \Z;\; b^2-4ac>0;\; p\nmid a, p|b, p|c    \},
$$
and let $D(Q):=b^2-4ac$.

\vspace{4mm}

\noindent
{\bf Theorem.}
{\em Given a suitable rigid cocycle $J$, we can construct a half-integral weight modular form $\mathcal{RS}(J)$ as  
   $$
    \mathcal{RS}(J):=\sum_{Q\in\Fp/\Gamma_0(p)} \hat{I}_k(J,Q) q^{D(Q)/p},
    $$
    where the coefficient $\hat{I}_k(J,Q)$ is obtained by pairing $Q$ with a certain residue polynomial of $J$. Here $q:=e^{2\pi i \tau}$ with $\tau\in\mathcal{H}$. The assignment $J \mapsto \mathcal{RS}(J)$ is Hecke-equivariant away from $p$. 
}

\vspace{4mm}

The residue polynomials mentioned in the statement will be defined in Section \ref{Rigid cocycles}, here we will limit ourselves to mention that they are $p$-adic counterparts for period polynomials of modular forms. A more precise, though still not completely accurate, statement of the theorem above will be given in Section \ref{The Shintani lift}. For the fully rigorous statement, see Theorem \ref{main thm}.
Note that $\mathcal{RS}(J)$ is a classical modular form but its coefficients are written in terms of $J$, a $p$-adic object.

The series above gives the \emph{rigid Shintani} map of the cocycle $J$, which is an analogue of the classical Shintani lift $\mathcal{S}$ of \cite{Shintani}. The map $\mathcal{S}$ associates a modular form of weight $k+1/2$ to a modular form of weight $2k$, so in our analogy rigid cocycles play the role of integral weight modular forms. Moreover, the functions $\mathcal{S}$ and $\mathcal{RS}$ have a similar construction: indeed, given $f$ of weight $2k$, the Fourier coefficients of the Shintani lift $\mathcal{S}(f)$ are given by period integrals of $f$. This analogy between residues of rigid cocycles and periods of modular forms was already exploited in \cite{PhD}, where an analogue for rigid cocycles of the Shimura lift of \cite{Shimura} was contructed by providing a theta kernel function with coefficients in rigid cocycles. The rigid analytic analogues of \cite{PhD} and of the present work behave similarly to their classical counterparts, and their constructions shed light on some parallels between the classical setting and that of rigid cocycles. In particular the coefficients of the modular form $\mathcal{RS}(J)$ share a striking similarity with the ones of the Shintani lift $\mathcal{S}(f)$. For more details on the analogy between classical and rigid analytic Shimura and Shintani lifts, see Section \ref{Classical and rigid analytic Shimura-Shintani lifts}.

Both the present work and \cite{PhD} fit into the $p$-adic Kudla program, the ideas of which were already present in \cite{DV2}, as they construct modular generating series whose coefficients arise from rigid cocycles, and use them to define maps which are counterparts to theta correspondences in this new setting.

\vspace{2mm}

\noindent
\textbf{Acknowledgements.} We are grateful to Vinayak Vatsal for many insightful conversations, and to Stephen Kudla for for his thorough review of earlier drafts of this paper, which resulted in several improvements in the presentation. We also thank Alice Pozzi for helpful comments on the exposition.
This paper is based upon work supported by the National Science
Foundation under Grant No. DMS-1928930 while the author was in
residence at the Mathematical Sciences Research Institute in Berkeley,
California, during the Spring 2023 semester.


\section{Motivation and background}
\label{Notation and preliminary definitions}
\subsection{Rigid cocycles}\label{Rigid cocycles}
Let $\Hp$ and $\Gamma$ be as defined in the Introduction, and let $k\geq 0$ be an integer. We will denote by $\mathcal{A}_{2k}$ the additive group of rigid analytic functions on $\Hp$, the Drinfeld's $p$-adic upper-half plane, with the weight $2k$ action of the group $\Gamma$ defined as
$$ f|\gamma (z) = (cz+d)^{-k} f\left(\frac{az+b}{cz+d}\right), 
\quad \mbox{ for }f\in\mathcal{A}_{2k} \:\text{ and }\: \gamma = \begin{pmatrix}a & b \\c & d \end{pmatrix}\in \Gamma.$$. For more details on $\Hp$ and rigid functions see \cite{Das}. 

\vspace{4mm}

\begin{definition}
\label{original def}
A weight $2k$ rigid analytic cocycle is an element of the $\C_p$-vector space $H^1_{par}(\Gamma,\mathcal{A}_{2k})$.
\end{definition}

\vspace{4mm}

Here $H^1_{par}(\Gamma,\mathcal{A}_{2k})\subset H^1(\Gamma,\mathcal{A}_{2k})$ is the parabolic cohomology of $\Gamma$, defined as usual (see for example \cite{Hida}).
Although the definition above agrees with the one of \cite{DV1}, in the present work we will adopt an equivalent definition given in terms of $\mathcal{A}_{2k}$-valued modular symbols. These are functions 
$$ m: \mathbb{P}_1(\Q) \times \mathbb{P}_1(\Q) \rightarrow  \mathcal{A}_{2k} $$ 
such that
$$
 m\{r,s\}=- m\{s,r\}\:\:\:\: \text{ and }\:\:\:\: m\{r,s\} + m\{s,t\} = m\{r,t\}, \qquad \mbox{ for all } r,s,t\in \mathbb{P}_1(\Q).
$$
Such a modular symbol is called $\Gamma$-invariant if
$$ 
m\{\gamma r, \gamma s\}|\gamma  = m\{r,s\}, \mbox{ for all } \gamma\in \Gamma.
$$
We will denote by $\RAC$ the $\C_p$ vector space of $\Gamma$-invariant modular symbols with values in $\mathcal{A}_{2k}$. Lemmas 1.3 and 1.9 of \cite{DV1} imply that there is an isomorphism between $H^1_{par}(\Gamma,\mathcal{A}_{2k})$ and $\RAC$, so we can give the following definition, which we will adopt for the rest of the paper, noting that it is much more explicit and convenient for computations.

\vspace{4mm}

\begin{definition}
A weight $2k$ rigid analytic cocycle is an element of the $\C_p$-vector space $\RAC$. This definition is equivalent to Definition \ref{original def}.
\end{definition}

\vspace{4mm}

We will now define a rational structure on $\RAC$ using the \emph{residues} of rigid cocycles. This is an anlogue of the rational structure on classical modular forms defined via their periods (see \cite{KZ}). We need to recall from \cite{PhD} the definition of the residue map on $\RAC$. Let $\mathcal{P}_{2k-2}$ be the homogeneous polynomials with coefficients in $\C_p$ and degree $2k$, with the action of $\Gamma$ defined as
$$
(h | \gamma ) (X,Y):=h(aX+bY,cX+dY), \:\:\:\:\:\:\text{for }h\in \Polys, \: \gamma= \begin{pmatrix}a&b\\c&d\end{pmatrix}\in \Gamma.
$$
Denote by $\MSPoly$ the space of $\Gamma_0(p)$-invariant modular symbols with values in $\Polys$. Let $\T$ be the Buhat-Tits tree of $\PGL_2(\Q_p)$. This object can be seen as a \emph{skeleton} for $\Hp$ by means of a $\PGL_2(\Q_p)$-equivariant reduction map $red: \Hp \rightarrow \T$. We will not use any property of $\T$ or the reduction map in the present work, but we will sometimes refer to the edges of $\T$. In particular, $e_0$ will denote the \emph{standard edge} of the Bruhat-Tits tree. Given $f$ a rigid function and $e$ an edge of $\T$, one can define the \emph{annular residue} $\Res_e(f)\in \C_p$ of $f$ at $e$ using the reduction map from $\Hp$ to $\T$. For more on $\T$ and its relationship with $\Hp$, see \cite{Das}.

\vspace{4mm}

\begin{definition}
\label{res map}
The residue map $\Res_0:\RAC \rightarrow \MSPoly$ is defined as $\Res_0(J):=\mu$, where
\begin{equation}
    \label{explicit res}
    \mu\{r,s\}(X,Y):=\sum_{i=0}^{2k-2}{2k-2 \choose i}(-1)^{i}\Res_{e_0} (z^{2k-2-i}J\{r,s\}(z))X^iY^{2k-i}.
\end{equation}
\end{definition}

\vspace{4mm}

Note that expression (\ref{explicit res}) will never be used in this paper and we only included it to provide some background. We are now ready to define a rational structure on $\RAC$.

\vspace{4mm}

\begin{definition}
We denote by $\RACQ$ the rigid analytic cocycles $J$ such that $\Res_0(J)$ is a modular symbol with values in polynomials with coefficients in $\bar{\Q}$. In other words,  $J\in\RACQ$ if and only if the numbers $\Res_{e_0} (z^{2k-2-i}J\{r,s\}(z))$ are in $\bar{\Q}$ for $i=0,...,2k-2$.
\end{definition}

\vspace{4mm}

Note that $\RACQ$ is not empty, as it contains the cocycles $J_{k,D}$ of \cite{PhD} (see \emph{loc. cit.}, Theorem 4.1).

\vspace{2mm}

We will now define the space of \emph{harmonic cocycles} on the edges of $\T$, which is used in some proofs of Section \ref{Hecke operators} and was used in \cite{PhD} to give an alternate decription of $\RAC$. For any $\Gamma$-module $\Omega$, the harmonic cocycles $C_{har}(\Omega)$ are $\Omega$-valued functions on the edges of $\T$ satisfying certain \emph{harmonicity conditions}. More precisely any $c \in C_{har}(\Omega)$ is such that
$$
\sum_{s(e)=v} c(e)=0,
$$
where the sum is taken over all the edges $e$ having the same vertex $v=s(e)$ as their starting point, and
$$
c(\bar{e})=-c(e),
$$
where $\bar{e}$ is the edge obtained by flipping the orientation of $e$ (i.e. by switching the starting and ending points of $e$). The action of $\Gamma$ on $C_{har}(\Omega)$ is given by
$$
(c|\gamma)(e):=c(\gamma e)|\gamma.
$$
For any edge of $\T$, the evaluation map $C_{har}(\Omega)\rightarrow \Omega$ is defined as $ev_e(c):=c(e)$. We will work with harmonic cocycles with $ \Omega=\Polys$, and more precisely with modular symbols values in such harmonic cocycles.

\vspace{4mm}

\begin{definition}
    We will denote by $\MS^{\Gamma}(C_{har}(\Polys))$ the space of $\Gamma$-invariant modular symbols with values in $C_{har}(\Polys)$. 
\end{definition}

\vspace{4mm}

It was shown in \cite{PhD} that there is an isomorphism $ST:\MS^{\Gamma}(C_{har}(\Polys))\rightarrow\RAC$, called the Schneider-Teitelbaum lift. Its inverse $\Res$ is the generalized residue map, which we do not need to define here. We will limit ourselves to show the following diagram for the sake of clarity
\begin{equation}
\label{clarity}
    \RAC \xrightarrow{Res} \MS^{\Gamma}(C_{har}(\Polys))\xrightarrow{ev_{e_0}}\MS^{\Gamma_0(p)}(\Polys),
\end{equation}
where, with a slight abuse of notation, $ev_{e_0}$ is the map induced on modular symbols by the $\Gamma_0(p)$-equivariant evaluation of harmonic cocycles at the edge $e_0$.

\vspace{4mm}

\begin{remark}
The residue map $\Res_0$ of Definition \ref{res map} can be seen as $\Res_0=ev_{e_0}\circ \Res$ (see Section 4 of \cite{PhD} for more details). This will be used in the proof of Lemma \ref{technical lemma}.
\end{remark}

\vspace{4mm}

Finally, we can consider a different action on $\Polys$, given by
$$
(h \star \gamma ) (X,Y):=h(dX-cY,-bX+aY), \:\:\:\: \text{ with }\gamma= \begin{pmatrix}a&b\\c&d\end{pmatrix}\in \Gamma.
$$
Note that $h \star \gamma=h|(\gamma^T)^{-1}$ and let $S:= \left(\begin{smallmatrix}0&1\\-1&0\end{smallmatrix}\right)$. Then $(\gamma^T)^{-1}=S^{-1}\gamma S$ and we can pass from one action to the other with the following $\Gamma$-equivariant map
\begin{eqnarray*}
    \alpha: (\Polys, |)&\rightarrow& (\Polys, \star)\\
    h &\mapsto & h|S.
\end{eqnarray*}
We consider both actions in order to be consistent with the literature: indeed \cite{Shintani} and \cite{Stevens} use $\star$ while the action used in \cite{PhD} and in the literature on rigid cocycles is the $|$ one. Consequently, the proofs involving polynomials in Section \ref{Construction of map} use the $\star$ action, while the definitions and proofs in Section \ref{Hecke operators} use the $|$ one. This is not an issue, as the map $\alpha$ is used in Section \ref{Construction of map} to phrase everything in the $\star$ notation.

As $\alpha$ is $\Gamma$-equivariant, it induces a map on $\MSPoly$, so let $\widetilde{\Res}_0:=\alpha \circ \Res_0$. We will often adopt the notation $\mu_J:=\Res_0(J)$ and $\widetilde{\mu}_J:=\widetilde{\Res}_0(J)$.

\vspace{4mm}

\subsection{Classical and rigid analytic Shimura-Shintani lifts}\label{Classical and rigid analytic Shimura-Shintani lifts}

In this Section we are going to recall some facts about the classical Shimura and Shintani lifts, and compare them to their analogues in the framework of rigid cocycles.


\subsubsection{The Shimura lift.}
\label{The Shimura lift}

Let $k$ be an even integer. Given a half-integral weight modular form $g\in S_{k+1/2}(\Gamma_0(4))$, Shimura (\cite{Shimura}) constructed an integral weight form $\mathcal{SH}(g)\in S_{2k}(\SL_2(\Z))$. The latter can be obtained by pairing $g$ with a theta kernel function $\Omega_k$ via the Petersson inner product. More precisely, $\Omega_k$ is a function of two variables $z,\tau\in\mathcal{H}$, given by
\begin{equation}
\label{shimura kernel}
    \Omega_k(z,q) := \sum_{D>0} D^{k-1/2} f_{k,D}(z) q^{ D},
\end{equation}
where $q:=e^{2\pi i \tau}$ and $f_{k,D}\in S_{2k}(\SL_2(\Z))$ are the so-called \enquote{Zagier forms} defined in \cite{Za}. The series $\Omega_k$ is a half-integral weight modular form with coefficients in $S_{2k}(\SL_2(\Z))$, meaning that $\Omega_k\in S_{k+1/2}(\Gamma_0(4))\otimes S_{2k}(\SL_2(\Z))$.  In simpler terms, the series $\Omega_k$ belongs to $S_{k+1/2}(\Gamma_0(4))$ as a function of $\tau$. Then, up to a multiplicative constant,
\begin{equation}
\label{Pet}
    \mathcal{SH}(g)=\langle g, \Omega_k \rangle_{Pet}=\int_{\Gamma_0(4)\backslash\mathcal{H}} g(\tau)\overline{\Omega_k(-\overline{z},\tau)}v^{k-3/2}dudv, \quad \tau=u+iv.
\end{equation}
In order to compare $\mathcal{SH}$ with its rigid analytic analogue, it will be useful to give a different expression for (\ref{Pet}). Let $\{g_j\}_{j=1,...,N}$ be an orthonormal basis for $S_{k+1/2}(\Gamma_0(4))$, and let $g=\sum_{j=1}^{N}\alpha_jg_j$ and $\Omega_k=\sum_{j=1}^{N}g_j\otimes f_j$ with $\alpha_j\in \C$ and $f_j\in S_{2k}(\SL_2(\Z))$. Then we can write
\begin{equation}
\label{pairing shimura}
    \mathcal{SH}(g)=\langle g, \Omega_k \rangle_{Pet}=\sum_{j=1}^{N}\alpha_jf_j.
\end{equation}
In \cite{PhD}, an analogue $\mathcal{RSH}$ of $\mathcal{SH}$ was defined in the setting of rigid cocycles. More precisely, let $k$ be odd now, then $\mathcal{RSH}: S_{k+1/2}^{(\bar{\Q})}(\Gamma(4p^2))\rightarrow \RAC$, where $S_{k+1/2}^{(\bar{\Q})}(\Gamma(4p^2))$ are the modular forms of weight $k+1/2$ and level $4p^2$ whose Fourier coefficients are in $\bar{\Q}$. The key ingredient for the construction of $\mathcal{RSH}$ is a theta kernel function, i.e. a half-integral weight modular form with coefficients in rigid cocycles. To define this, begin with the formal $q$-series
\begin{equation}
\label{Rig shimura kernel}
    \hat{\Omega}_k(q) := \sum_{D>0} D^{k-1/2} J_{k,D}q^{ D} \in \RAC\big[\big[q\big]\big],
\end{equation}
where $J_{k,D}\in\RAC$ are $p$-adic analogues of the Zagier forms (see Section 2 of \cite{PhD} for their explicit definition). Consider the natural map
$$
\Psi: S_{k+1/2}^{(\bar{\Q})}(\Gamma(4p^2))\otimes_{\bar{\Q}} \RAC \rightarrow \RAC\big[\big[q\big]\big].
$$
The main result of \cite{PhD} is that 
\begin{equation}
    \hat{\Omega}_k\in \Psi \big ( S_{k+1/2}^{(\bar{\Q})}(\Gamma(4p^2))\otimes_{\bar{\Q}} \RAC\big).
\end{equation}
For more details see Theorems 6.1, 6.2, 6.3 in \emph{loc. cit.} The preimage under $\Psi$ of the series $\hat{\Omega}_k(q)$ will be denoted by $\hat{\Omega}_k(\tau)$, where $\tau \in \mathcal{H}$.
Let now $\{g_j\}_{j=1,...,M}$ be a basis of $S_{k+1/2}^{(\bar{\Q})}(\Gamma(4p^2))$ and write 
\begin{equation}
    \hat{\Omega}_k(\tau)=\sum_{j=1}^{M}g_j(\tau)\otimes m_j
\end{equation}
with $m_j\in \RAC$. Now let $g:=\sum_{j=1}^{M}\alpha_jg_j\in S_{k+1/2}^{(\bar{\Q})}(\Gamma(4p^2))$ with $\alpha_j\in\bar{\Q}$. Then in light of (\ref{pairing shimura}) it is natural to define
\begin{equation}
\label{pairing rig shimura}
    \mathcal{RSH}(g):=\sum_{j=1}^{M}\alpha_j m_j \in \RAC.
\end{equation}
Comparing (\ref{shimura kernel}) and (\ref{Rig shimura kernel}) as well as (\ref{pairing shimura}) and (\ref{pairing rig shimura}), we see that in our analogy rigid analytic cocycles play the role of modular forms.

\subsubsection{The Shintani lift.}
\label{The Shintani lift} In \cite{Shintani} Shintani constructed a map going from integral weight modular forms to half-integral weight forms. We are now going recall the classical Shintani lift for level $p$ forms and compare it with our main result.

Let $\chi$ be the Dirichlet character modulo $4p$ defined as
\begin{equation}
\label{char}
   \chi(d):=\Big( \frac{(-1)^{k}p}{d} \Big), \;\;d\in (\Z/4p\Z)^{\times}. 
\end{equation}
Given a binary quadratic form $Q(x,y):=ax^2+bxy+cy^2$, we will adopt the notation $Q:=[a,b,c]$. The discriminant of $Q$ will be denoted by $D(Q)$. Let $\Fp$ be the set of binary quadratic forms defined as
\begin{equation}
\label{BQFs}
\Fp:=\{[a,b,c]|\; a,b,c \in \Z;\; D(Q)>0;\; p\nmid a, p|b, p|c    \}.
\end{equation}
A matrix $\left(\begin{smallmatrix}\alpha &\beta\\\gamma&\delta\end{smallmatrix}\right) \in \Gamma$ acts on $Q(x,y)$ by
$$
(Q\star\gamma)(x,y):=Q(\delta x- \gamma y, -\beta x+ \alpha y).
$$

This action is compatible with the $\star$ action on polynomials defined in Section \ref{Rigid cocycles}, and in the present section the action on $\Polys$ is the $\star$ one. 
Note that $\Gamma_0(p)$ preserves $\Fp$. The stabilizer $\Gamma_Q$ of $Q$ in $\Gamma$ is generated up to torsion by the \emph{automorph} $\gamma_Q$ of $Q$, and $\gamma_Q\in\Gamma_0(p)$ if $Q\in\Fp$ (see \cite{Shintani}). We recall now the definition of the \emph{Shintani cycle} $C_Q$ associated to $Q\in\Fp$, following \cite{Shintani}.

\begin{definition}
    Let $\omega$ be an arbitrary point in $\PQ$. The Shintani cycle associated to $Q:=[a,b,c]\in\Fp$ is the oriented geodesic $C_Q:=(r_Q,s_Q)$ in the upper half-plane, where
    $$
(r_Q,s_Q)= \left\{
\begin{array}{ll}
      (x_Q, y_Q) & \text{if } D(Q) \text{ is a square},  \\
      (\omega,\gamma_Q(\omega)) & \text{otherwise},\\
\end{array} 
\right.
$$
with
$$
(x_Q,y_Q)= \left\{
\begin{array}{ll}
       \Big(\frac{b+\sqrt{D_Q}}{2c}, \frac{b-\sqrt{D_Q}}{2c}\Big) & \text{if } c \ne 0,  \\
       \big(\infty,\frac{a}{b}\big) & \text{if } c=0,\:b>0,\\
       \big(\frac{a}{b}, \infty\big) & \text{if } c=0,\:b<0.\\
\end{array} 
\right.
$$
\end{definition}

\vspace{4mm}

\begin{theorem}(Shintani)
\label{Shintani thm}
    Let $f\in S_{2k}(\Gamma_{0}(p))$ and let $\chi$ be as in (\ref{char}).
    Let
    $$
    I_k(f,Q):=\int_{C_Q}f(\tau)Q(1,-\tau)^{k-1}d \tau, \quad\text{with } Q\in \Fp.
    $$
    Then 
    $$
    \mathcal{S}(f):=\sum_{Q\in\Fp/\Gamma_0(p)}I_k(f,Q)q^{D(Q)/p}
    $$
    is a cusp form in $S_{k+1/2}(\Gamma_{0}(4p),\chi)$. Here $q:=e^{2\pi i \tau}$ with $\tau\in\mathcal{H}$. The map 
    $$
    \mathcal{S}: S_{2k}(\Gamma_{0}(p)) \rightarrow S_{k+1/2}(\Gamma_{0}(4p),\chi)
    $$
    is $\bar{\mathbb{C}}$-linear and $ \mathcal{S}(T_m(f))=T_{m^2}(\mathcal{S}(f)) $ for any $m$ coprime with $p$.
\end{theorem}

\vspace{4mm}

We will now state our main result. Although more precise than the one in the Introduction, the statement below is still paraphrasing certain points. A more rigorous statement will be given in Theorem \ref{main thm}, after having introduced all the necessary notation. Let $\langle\:,\: \rangle$ be the pairing on $\mathcal{P}_{2k-2}$ defined as
$$
\langle X^iY^{2k-2-i}, X^jY^{2k-2-j}   \rangle={2k-2 \choose i}^{-1}(-1)^i \delta_{i,2k-2-j}.
$$

\vspace{4mm}

\noindent
{\bf Main Theorem.}
{\em Let $J\in\RAC$ be a suitable rigid cocycle and let $\chi$ be the character defined in (\ref{char}).  Let
    $$
    \hat{I}_k(J,Q):=\langle \widetilde{\mu}_J\{r_Q,s_Q\}(x,y),Q^{k-1}(x,y)\rangle, \quad\text{with } Q\in \Fp.
    $$Then
   $$
    \mathcal{RS}(J):=(2k-2)!\sum_{Q\in\Fp/\Gamma_0(p)} \hat{I}_k(J,Q) q^{D(Q)/p}
    $$
    is a cusp form in $S_{k+1/2}(\Gamma_{0}(4p),\chi)$. Here $q:=e^{2\pi i \tau}$ with $\tau\in\mathcal{H}$. The assignment $J \mapsto \mathcal{RS}(J)$ is Hecke-equivariant away from $p$. 
}

\vspace{4mm}

To see the similarity between $I_k(f,Q)$ and $\hat{I}_k(J,Q)$, consider the period map 
$$
f\mapsto \widetilde{per}(f):=\int_r^s f(z)(zX+Y)^{2k-2}dz,
$$
which will be further described in the next Section.

Then $I_k(f,Q)$ can be characterized as
\begin{equation}
\label{similar expression}
    I_k(f,Q)=(2k-2)!\langle (\widetilde{per}(f)\{r_{Q},s_{Q}\})(x,y),Q^{k-1}(x,y)\rangle.
\end{equation}
See for example Section 4.3 of \cite{Stevens} for more details.  Comparing (\ref{similar expression}) with $\hat{I}_k(J,Q)$ in the Main Theorem, we see that the only difference between them is that in the rigid analytic case the residues $\widetilde{\mu}_J$ of $J$ take the place of the periods of modular forms. (The factor $(2k-2)!$ is not important, as it simply arises because the pairing $\langle \:, \rangle$ that we chose is slightly different from the one used by Shintani). This parallel between periods of modular forms and residues of rigid cocycles already appeared in \cite{PhD}, where the residues of the cocycles $J_{k,D}$ of (\ref{shimura kernel}) were shown to be equal to the periods of certain modular forms related to the Zagier forms $f_{k,D}$ of (\ref{Rig shimura kernel}) (see Sections 4 and 5 of \cite{PhD}).

\subsection{Eichler-Shimura isomorphism and cuspidal cocycles}\label{Eichler-Shimura and cuspidal cocycles}
We briefly recall some facts about the Eichler-Shimura isomorphism and introduce some notation that will be needed to prove our main result in Section \ref{Construction of map}. For a similar treatment of this topic and for more details see Section 4 of \cite{Stevens}.

Consider the involution on $\MS^{\Gamma_0(p)}(\mathcal{P}_{2k-2})$ given by
$$
(m|w_{\infty})\{r,s\}:=(m\{w_{\infty}r,w_{\infty}s\})|w_{\infty},\:\:\text{with }w_{\infty}:=\begin{pmatrix}-1 & 0\\ 0 &1\end{pmatrix}.
$$
We will call \emph{even} (resp. \emph{odd}) elements of $\MSPoly$ which are in the $+1$ (resp. $-1$) eigenspace for $w_{\infty}$.
Every modular symbol in $\MSPoly$ can be written as a sum of an even and an odd modular symbol, and we have
\begin{equation}
    \label{direct sum}
\MSPoly=\MSPoly^{+}\oplus\MSPoly^-.
\end{equation}
Clearly the same definitions could be given using the $\star$ action.

We will denote by $\Polys^{(\C)}$ the homogeneous polynomials with coefficients in $\C$ and degree $2k-2$, with the $|$ action of $\Gamma_0(p)$. Recall the Eichler-Shimura isomorphism, in particular the diagram
\begin{center}
\begin{tikzcd}[column sep=large, column sep=large]
S_{2k}(\Gamma_0(p))\oplus S_{2k}(\Gamma_0(p))\arrow{r}{ES}  \arrow{rd}{per} 
  & H^1(\Gamma_0(p),\Polys^{(\C)})  \\
    & \MSPolyC \arrow{u}{\beta}
\end{tikzcd}
\end{center}

Here $ES$ is an isomorphism, $\beta$ associates to a modular symbol $\mu$ the cocycle $\Phi_{\mu}(\gamma):=\mu\{x_0,\gamma x_0\}$ for any $x_0\in\PQ$, and $per$ is given by
$$
f\mapsto per(f):=\int_r^s f(z)(X-zY)^{2k-2}dz.
$$
Moreover $per=per^+\oplus per^-$, where $per^+$ (resp. $per^-$) denotes the map given by taking even (resp. odd) periods of cusp forms. These maps land in the spaces of even and odd modular symbols, which are the eigenvalues of $w_{\infty}$ acting on $\MSPolyC$. By the Manin-Drinfeld principle there is a unique Hecke-equivariant section $s_{\beta}$ of $\beta$ such that the diagram commutes and $\MSPolyC=Ker(\beta)\oplus Im(s_{\beta})$. We have $s_{\beta}=per\circ ES^{-1}$.

If we consider $\Polys^{(\C)}$ with the $\star$ action instead, we get an analogue diagram
\begin{center}
\begin{tikzcd}[column sep=large, column sep=large]
S_{2k}(\Gamma_0(p))\oplus S_{2k}(\Gamma_0(p))\arrow{r}{\widetilde{ES}}  \arrow{rd}{\widetilde{per}} 
  & H^1(\Gamma_0(p),\Polys^{(\C)})  \\
    & \MSPolyC \arrow{u}{\beta}
\end{tikzcd}
\end{center}
Now the period map is given by 
$$
f\mapsto \widetilde{per}(f):=\int_r^s f(z)(zX+Y)^{2k-2}dz.
$$
The maps $\widetilde{ES}$ and $\widetilde{s}_{\beta}$ are defined like in the previous case. To relate the objects in \cite{PhD} with the ones of \cite{Shintani} and \cite{Stevens}, we will use the fact that $\widetilde{per}(f)=\alpha\circ per(f)$ and $\widetilde{per}^{\pm}(f)=\alpha\circ per^{\pm}(f)$.

\vspace{4mm}

If $J\in\RACQ$, then $\Res_0(J)$ can be seen as an element of $\MSPolyC$, hence the following definition makes sense.

\vspace{4mm}

\begin{definition}
    A rigid analytic cocycle $J\in\RACQ$ is called cuspidal if  $\Res_0(J)\in Im(s_{\beta})$ or equivalently if $\widetilde{\Res}_0(J)\in Im(\widetilde{s}_{\beta})$. The space of cuspidal cocycles will be denoted by $\RAC^{\text{cusp}}$.
\end{definition}

\vspace{4mm}

Note that the cocycles $J_{k,D}$ of \cite{PhD} are cuspidal as $\Res_0(J_{k,D})=per^-(f_{k,D}^{(p)})$, for certain cusp forms $f_{k,D}^{(p)}\in S_{2k}(\Gamma_0(p))$.

\vspace{4mm}

\section{Hecke operators on rigid analytic cocycles}\label{Hecke operators}

In this Section we will provide two different definitions of Hecke operators on rigid analytic cocycles: the first one will be given using the map $\Res_0:\RAC \rightarrow \MSPoly$ of \cite{PhD}, the second one will be more explicit. Proposition \ref{same def} will show that the two definitions coincide.

Let $\{\gamma_j\}_{j=0,...,p}$ be the representatives given by $\SL_2(\Z)=\bigsqcup_{j=0}^{p}\gamma_j \Gamma_0(p)$ and let $w_p:=\left(\begin{smallmatrix}0 & -1 \\p & 0 \end{smallmatrix}\right)$.

\vspace{4mm}

\begin{definition}
\label{p new}
    The space $\MSPoly^{p\text{-}new}$ is the subspace given by the modular symbols $\psi\in\MSPoly$ such that 
    $$
    \sum_{j=0}^{p}\psi|\gamma_j=0 \:\:\:\:\text{ and } \:\:\:\:\sum_{j=0}^{p}\psi|(w_p \gamma_j^{-1})=0.
    $$
\end{definition}

\vspace{4mm}

\begin{lemma}
\label{technical lemma}
    The map $\Res_0$ gives an isomorphism $\Res_0:\RAC \rightarrow \MSPoly^{p\text{-}new}$.
\end{lemma}
\begin{proof}
    Recall from (\ref{clarity}) that $\Res_0=ev_{e_0}\circ\Res$. As $\Res$ is the inverse of the Schneider-Teitelbaum lift $ST$, it must be injective. Moreover, $ev_{e_0}(c)$ belongs to $\MSPoly^{p\text{-}new}$ because of the harmonicity properties of $c$. Indeed, the conditions in Definition \ref{p new} can be rewritten as sums of $c$ evaluated at all the edges leaving the standard vertex $v_0$ of $\T$. So the only thing left to prove is that $ev_{e_0}$ is surjective on $\MSPoly^{p\text{-}new}$. Any edge $e$ of $\T$ can be written as $e=\gamma e_0$ for some $\gamma\in\Gamma$. Then, given $c_0\in\MSPoly^{p\text{-}new}$, we can define 
    $$
c\{r,s\}(e):=c_0\{\gamma r, \gamma s\}|\gamma.
    $$
    Using the fact that $c_0$ satisfies the formulas in Definition \ref{p new} one can show that $c\in \MS^\Gamma(C_{har}(\Polys))$. It is easy to check that $ev_{e_0}(c)=c_0$.
\end{proof}

\vspace{4mm}

Recall that there is an action of Hecke operators on $\MSPoly$, induced by the Hecke action on $S_{2k}(\Gamma_0(p))$ via the Eichler-Shimura isomorphism $ES$. For $J\in\RAC$ and any positive integer $m$ with $(m,p)=1$ we can now define Hecke operators $T_m$ as
\begin{equation}
\label{first def}
    T_m(J):=\Res_0^{-1}(T_m(\Res_0(J))).
\end{equation}
Similarly, let
\begin{equation}
\label{first def p}
    U_p(J):=\Res_0^{-1}(U_p(\Res_0(J))).
\end{equation}
Alternatively, we can give more concrete definitions as follows. Let $l$ be a prime different from $p$. Define the sets of matrices
\begin{equation}
\label{l matrici}
\Gamma_l:=\Big\{\begin{pmatrix}1 & a\\ 0&l
\end{pmatrix}, \: a=0,...,l-1 \Big\}\cup \Big\{ \begin{pmatrix}l & 0\\ 0&1
\end{pmatrix} \Big\}.
\end{equation}
Any $\gamma=\begin{pmatrix} a&b\\c&d\end{pmatrix}\in \Gamma_l$ acts on $f\in \mathcal{A}_{2k-2}$ by
$$
(f|\gamma)(z):=det(\gamma)^{k-1}(cz+d)^{2-2k}f(\gamma z),
$$
and on $h\in\Polys$ by
$$
(f|\gamma)(z):=det(\gamma)^{1-k}(cX+dY)^{2k-2}h(aX+bY, cX+dY).
$$
These induce actions of $\gamma$ on $\RAC$, $\MS^\Gamma(C_{har}(\Polys))$ and $\MSPoly$.

Then, we can define
\begin{equation}
\label{second def}
T_l(J):=l^{k-1}\sum_{\gamma \in\Gamma_l}J|\gamma.
\end{equation}
Note that the $T_l$'s defined in (\ref{second def}) preserve the space $\RAC$, by the same argument used to show that Hecke operators preserve the space of modular forms.
To define the operator $U_p$, consider the involution on $\RAC$ defined by 
$$
(J|w_{p})\{r,s\}:=(J\{w_{p}r,w_{p}s\})|w_{p},\:\:\text{with }w_{p}:=\begin{pmatrix}0 & -1\\ p &0\end{pmatrix},
$$
where $J\in\RAC$. Then let
\begin{equation}
\label{second def p}
    U_p(J):=-J|w_{p}.
\end{equation}

In order to prove that the definitions given in (\ref{first def}), (\ref{first def p}) and (\ref{second def}), (\ref{second def p}) coincide, we need some notation. 
As $w_p$ is an involution, we can write 
$$
\RAC= \RAC^{p,+}\oplus \RAC^{p,-},
$$
where $\RAC^{p,+}$ (resp. $\RAC^{p,-}$) is the eigenspace on which $w_{p}$ acts with eigenvalue $+1$ (resp. $-1$). The two subspaces that we just defined are called the spaces of $p$-\emph{even} and $p$-\emph{odd} rigid analytic cocycles. Similarly, as $\RAC \cong \MS^\Gamma(C_{har}(\Polys))$, we can define $\MS^\Gamma(C_{har}(\Polys))^{p,+}$ and $\MS^\Gamma(C_{har}(\Polys))^{p,-}$.

For any $c\in C_{har}(\Polys)$ and for any edge $e$ of the Bruhat-Tits tree, recall from Section \ref{Rigid cocycles}  that $ev_{e}$ is the evaluation map at $e$, i.e. $ev_e(c):=c(e)$. 

Finally, given the sign $\epsilon=+$ (resp. $\epsilon=-$), let 
$$
\MSPoly^{U_p=\epsilon}
$$ 
be the subspace of $\MSPoly$ on which $U_p$ acts as multiplication by $+1$ (resp. $-1$).

\vspace{4mm}

\begin{proposition}
    \label{translation}
    The map $ev_{e_0}$ induces Hecke-equivariant inclusions
    $$
    ev_{e_0}:\MS^\Gamma(C_{har}(\Polys))^{p, \epsilon}\hookrightarrow\MSPoly^{U_p=-\epsilon}.
    $$
\end{proposition}
\begin{proof}
    The proof is similar to the one of Proposition 3.3 of \cite{DV2}.
    
    For any edge $e$ of the Bruhat-Tits tree, let $U(e)\subset \PQp$ be the $p$-adic ball attached to $e$. Then the complement of $U(e_0)$ is 
    $$U(e_0)^{\prime}=\Z_p=\bigsqcup_{j=1,...,p} U(e_j),
    $$
    where
    $$
    U(e_j):=\alpha_j^{-1}\Z_p, \:\:\:\:\text{ with } \alpha_j:=\begin{pmatrix}1 & j-1\\ 0&p\end{pmatrix}.
    $$
    Let $c\in \MS^\Gamma(C_{har}(\Polys))^{p,\epsilon}$. The harmonicity of $c\{r,s\}$ implies that 
    $$
    (ev_{e_0}(c))\{r,s\}=-\sum_{j=1,...,p} (ev_{e_j}(c))\{r,s\}.
    $$
    From here we can proceed as \cite{DV2}, so the quantity above can be written as
    $$
    -\sum_{j=1,...,p} c\{r,s\}(e_j)=-\sum_{j=1,...,p}((c\{r,s\}|\alpha_j^{-1})(e_0))|\alpha_j=-\epsilon\sum_{j=1,...,p}(c\{\alpha_jr,\alpha_js\}(e_0))|\alpha_j.
    $$
    The result follows from the definition of $U_p$ given in (\ref{second def p}).
\end{proof}

\vspace{4mm}

\begin{proposition}
\label{same def}
    The definition of $T_l$ (resp. $U_p$) given in (\ref{first def}) (resp. (\ref{first def p})) coincides with the one given in (\ref{second def}) (resp. (\ref{second def p})).
\end{proposition}
\begin{proof}
The Schneider-Teitelbaum lift $\ST$ and its inverse $\Res$ are $\GL_2(\Q_p)$-equivariant, and the map $ev_{e_0}$ is equivariant for the matrices $\gamma\in\Gamma_l$ defined in (\ref{l matrici}) as those matrices stabilize $e_0$. This implies that the map $\Res_0=ev_{e_0}\circ \Res$ is $T_l$-equivariant for $T_l$ defined in (\ref{second def}), hence the definition in (\ref{first def}) agrees with the one in (\ref{second def}). Proposition \ref{translation} implies that $\Res_0$ is also equivariant for the operator $U_p$ as defined in (\ref{second def p}), so the result follows.
\end{proof}

\vspace{4mm}


\section{Construction of the Shintani map for rigid cocycles}
\label{Construction of map}
 
\vspace{4mm}

In this Section we will construct a map $\mathcal{RS}:\RAC^{cusp}\rightarrow S_{k+1/2}(\Gamma_{0}(4p),\chi)$, where $\chi$ was defined in (\ref{char}).

We can finally give a rigorous statement of the main theorem.

\vspace{4mm}

\begin{theorem}
\label{main thm}
    Let $J\in\RAC^{cusp}$ be a cuspidal cocycle and let $\chi$ be the character defined in (\ref{char}). Let
    $$
    \hat{I}_k(J,Q):=\langle \widetilde{\mu}_J\{r_Q,s_Q\}(x,y),Q^{k-1}(x,y)\rangle, \quad\text{with } Q\in \Fp.
    $$Then
   $$
    \mathcal{RS}(J):=(2k-2)!\sum_{Q\in\Fp/\Gamma_0(p)}\hat{I}_k(J,Q) q^{D(Q)/p}
    $$
    is a cusp form in $S_{k+1/2}(\Gamma_{0}(4p),\chi)$. Here $q:=e^{2\pi i \tau}$ with $\tau\in\mathcal{H}$. The map
    $$
    \mathcal{RS}: \RAC^{cusp} \rightarrow S_{k+1/2}(\Gamma_{0}(4p),\chi)
    $$
    is $\bar{\Q}$-linear and
    $$
    \mathcal{RS}(T_m(J))=T_{m^2}(\mathcal{RS}(J))
    $$
    for any $m$ coprime with $p$.
\end{theorem}

\vspace{4mm}

To prove Theorem \ref{main thm} we will need the following result.

\vspace{4mm}

\begin{lemma}
The quantity $\hat{I}_k(J,Q)$ is well defined and depends only on the $\Gamma_0(p)$-orbit of $Q$.
\end{lemma}
\begin{proof}
    We will start by proving that $\hat{I}_k(J,Q)$ does not depend on the choice of the point $\omega\in\PQ$ that we used to define it. Let $\wpr\in \PQ$. Then using the properties of $\Gamma_0(p)$-invariant modular symbols we get 
    \begin{eqnarray*}
    \widetilde{\mu}_J\{\omega,\gamma_Q(\omega)\}&=&\widetilde{\mu}_J\{\wpr,\gamma_Q(\wpr)\}+\widetilde{\mu}_J\{\gamma_Q(\wpr),\gamma_Q(\omega)\}-\widetilde{\mu}_J\{\wpr,\omega\}\\
    &=&\widetilde{\mu}_J\{\wpr,\gamma_Q(\wpr)\}+\widetilde{\mu}_J\{\wpr,\omega\}\star\gamma_Q^{-1}-\widetilde{\mu}_J\{\wpr,\omega\}.
    \end{eqnarray*}
    Now note that
    \begin{eqnarray*}
    \langle (\widetilde{\mu}_J\{\wpr,\omega\}\star\gamma_Q^{-1})(x,y),Q^{k-1}(x,y)\rangle&=&\langle \widetilde{\mu}_J\{\wpr,\omega\}(x,y),(Q\star\gamma_Q)^{k-1}(x,y)\rangle\\
    &=&\langle \widetilde{\mu}_J\{\wpr,\omega\}(x,y),Q^{k-1}(x,y)\rangle.
    \end{eqnarray*}
    From this one can immediately conclude that $\hat{I}_k(J,Q)$ does not depend on $\omega$ or $\wpr$. 

    We will now show that $\hat{I}_k(J,Q)=\hat{I}_k(J,Q\star\gamma)$ for any $\gamma\in\Gamma_0(p)$. At first note that $\hat{I}_k(J,Q\star\gamma)$ can be written as
     \begin{eqnarray}
     \label{test1}
    \langle \widetilde{\mu}_J\{r_{Q\star\gamma},s_{Q\star\gamma}\},(Q\star\gamma)^{k-1}(x,y)\rangle&=&\langle (\widetilde{\mu}_J\{r_{Q\star\gamma},s_{Q\star\gamma}\}\star\gamma^{-1})(x,y),Q^{k-1}(x,y)\rangle\\
    \label{test}
    &=&\langle \widetilde{\mu}_J\{\gamma (r_{Q\star\gamma}),\gamma (s_{Q\star\gamma})\}(x,y),Q^{k-1}(x,y)\rangle.
     \end{eqnarray}
     We are going to treat separately the cases in which $D(Q)$ is or is not a square. Let us assume at first that $D(Q)$ is not a square, i.e. $\{r_{Q\star\gamma},s_{Q\star\gamma}\}=\{\omega,\gamma_{Q\star\gamma}(\omega)\}$ for some $\omega\in\PQ$. The automorph of $Q\star\gamma$ is $\gamma_{Q\star\gamma}=\gamma^{-1}\gamma_Q\gamma$, so the pair appearing in (\ref{test}) is $\{\gamma (r_{Q\star\gamma}),\gamma (s_{Q\star\gamma})\}=\{\gamma\omega,\gamma_{Q}(\gamma\omega)\}$. This implies that the quantity in (\ref{test}) is equal to $\hat{I}_k(J,Q)$.

     We will now consider the case in which $D(Q)$ is a square. In this case the Shintani cycle $(r_Q,s_Q)$ is defined as $(g(\infty) ,g(0))$ for $g\in\SL_2(\R)$ such that $(Q\star g)(x,y)=\sqrt{D}xy$ (see \cite{Shintani},page 101). If $g$ is such a matrix, then $(Q\star\gamma)\star(\gamma^{-1}g)(x,y)=\sqrt{D}xy$, hence $(r_{Q\star\gamma},s_{Q\star\gamma})=(\gamma^{-1}g(\infty) ,\gamma^{-1}g(0))=(\gamma^{-1}r_Q, \gamma^{-1}s_Q)$. This fact, together with equations  (\ref{test1}), (\ref{test}) implies that also in this case $\hat{I}_k(J,Q)=\hat{I}_k(J,Q\star\gamma)$.
\end{proof}

\vspace{4mm}

Expression (\ref{similar expression}) implies that we can write $I_k(f,Q)=I_k(f,Q)^+ +I_k(f,Q)^-$, where
$$
I_k(f,Q)^{\pm}=(2k-2)!\langle (\widetilde{per}^{\pm}(f)\{r_{Q},s_{Q}\})(x,y),Q^{k-1}(x,y)\rangle.
$$
We will now prove that the even part $I_k(f,Q)^{+}$ of $I_k(f,Q)$ does not contribute to the Shintani lift $\mathcal{S}(f)$.

\vspace{4mm}

\begin{proposition}
\label{odd cycle}
    Let $f\in S_{2k}(\Gamma_{0}(p))$. Then 
    $$
  \sum_{Q\in\Fp/\Gamma_0(p)}I_k(f,Q)^+q^{D(Q)/p}=0,
    $$
    and
    $$
    \mathcal{S}(f)=\sum_{Q\in\Fp/\Gamma_0(p)}I_k(f,Q)^-q^{D(Q)/p}.
    $$
\end{proposition}
\begin{proof}
    Recall that $\tI=\left(\begin{smallmatrix} -1 & 0\\ 0 & 1\end{smallmatrix}\right)$ and let $Q^{\prime}:=Q|\tI$ for $Q\in\Fp$. Then 
    $$
    \widetilde{per}^{\pm}(f)=\frac{1}{2}\Big(\widetilde{per}(f)\:\pm\: \widetilde{per}(f)|\tI\Big),
    $$
    and
    $$
    I_k(f,Q)^{\pm}=\frac{1}{2}\Big(I_k(f,Q)\:\pm\: (2k-2)!\langle (\widetilde{per}(f)\{-r_{Q},-s_{Q}\})(x,y),{Q^{\prime}}^{k-1}(x,y)\rangle\Big).
    $$
    Using the definition of the Shintani cycle one can show that $\{-r_{Q},-s_{Q}\}=\{s_{Q^{\prime}},r_{Q^{\prime}}\}$, hence
    $$
    I_k(f,Q)^{\pm}=\frac{1}{2}\Big(I_k(f,Q) \:\mp\: I_k(f,Q^{\prime})\Big).
    $$
    So $I_k(f,Q)^+$ does not affect the sum giving $\mathcal{S}(f)$, as the contributions from $I_k(f,Q)$ and $I_k(f,Q^{\prime})$ cancel out.
\end{proof}

\vspace{4mm}

Now we can finally prove our main result.

\begin{proof}(Proof of Theorem \ref{main thm})

    As the cocycle $J$ is cuspidal, there are $(g,f)\in S_{2k}(\Gamma_0(p))\oplus S_{2k}(\Gamma_0(p))$ such that $\mu_J=per^+(g_J)+per^-(f_J)$, or equivalently $\widetilde{\mu}_J=\widetilde{per}^+(g_J)+\widetilde{per}^-(f_J)$. By Theorem \ref{Shintani thm} and Proposition \ref{odd cycle} we have
    $$
    \mathcal{RS}(J)=\mathcal{S}(f_J)
    $$
    and the result follows.
\end{proof}

\vspace{4mm}

For $J$ a cuspidal cocycle in $\RACQ$, write $\widetilde{\mu}_J=\widetilde{\mu}_J^++\widetilde{\mu}_J^-$ as a sum of an even and an odd modular symbol. Then the proof above has the following consequences.

\vspace{4mm}

\begin{corollary}
    Let $$
    \mathcal{RS}^{\pm}(J):=(2k-2)!\sum_{Q\in\Fp/\Gamma_0(p)}\langle \widetilde{\mu}_J^{\pm}\{r_Q,s_Q\}(x,y),Q(x,y)^{k-1}\rangle q^{D(Q)/p}.
    $$
    Then
    $$
    \mathcal{RS}(J)=\mathcal{RS}^{-}(J)\:\:\:\text{ and }\:\:\:\mathcal{RS}^{+}(J)=0.
    $$
\end{corollary}

\vspace{4mm}

\begin{corollary}
Let $\widetilde{s_{\beta}}^{-}$ be the map given by composing $\widetilde{s}_{\beta}\circ \widetilde{Res}_0$ with the projection on the second component of $S_{2k}(\Gamma_0(p))\oplus S_{2k}(\Gamma_0(p))$. Then $\mathcal{RS}$ factors through $S_{2k}(\Gamma_0(p))$ via $\widetilde{s_{\beta}}^{-}$.
\end{corollary}

\vspace{4mm}

\begin{example}
We can apply our main theorem to the rigid cocycle $J_{k,D}$ of equation (\ref{Rig shimura kernel}). It was shown in \cite{PhD} that $\mu_{J_{k,D}}=per^-(f_{k,D}^{(p)})$, where $f_{k,D}^{(p)}\in S_{2k}(\Gamma_0(p))$ is a level $p$ analogue of the Zagier form $f_{k,D}$ of equation (\ref{shimura kernel}) (see \cite{PhD}, Section 5). Hence $\mathcal{RS}(J_{k,D})=\mathcal{S}(f_{k,D}^{(p)})$.
\end{example}

\end{document}